\documentclass[12pt,reqno]{amsart}
\usepackage{amsmath, amsthm, amssymb}
\usepackage{mathtools}
\usepackage{stmaryrd}

\usepackage[backend=bibtex,firstinits=true,style=alphabetic,natbib=true,url=false,sorting=nyt,doi=true,abbreviate=true,backref=false]{biblatex}
\renewbibmacro{in:}{\ifentrytype{article}{}{\printtext{\bibstring{in}\intitlepunct}}}

\usepackage[margin=1in]{geometry}
\usepackage{parskip}
\usepackage[shortlabels]{enumitem}

\usepackage{color}

\usepackage[bookmarks,colorlinks,breaklinks]{hyperref}
\hypersetup{linkcolor=red,citecolor=black,
  filecolor=dullmagenta,urlcolor=darkblue}

\usepackage{cleveref}

\usepackage{graphicx}
\usepackage{tikz}
\usetikzlibrary{matrix, arrows}
\usetikzlibrary{cd}

\makeatletter
\def\thm@space@setup{%
  \thm@preskip=0.5em\thm@postskip=\thm@preskip%
}
\makeatother

\newtheoremstyle{named}{}{}{\\itshape}{}{\bfseries}{.}{.5em}{\thmnote{#3's }#1}
\theoremstyle{named}

\theoremstyle{plain}
\newtheorem{thm}{Theorem}[section]
\newtheorem{prop}[thm]{Proposition}
\newtheorem{lem}[thm]{Lemma}
\newtheorem{cor}[thm]{Corollary}
\theoremstyle{definition}

\newtheorem{exmpl}[thm]{Example}

\theoremstyle{remark}

\crefformat{footnote}{#2\footnotemark[#1]#3}

\newcommand{\Hom}{\mathrm{Hom}}

\newcommand{\CC}{\mathbb{C}}
\newcommand{\QQ}{\mathbb{Q}}

\newcommand{\ZZ}{\mathbb{Z}}

\newcommand{\GL}[1]{\mathrm{GL}_{#1}}

\newcommand{\tr}{\mathrm{tr}}

\newcommand{\cM}[1]{\mathcal{M}_{#1}}

\newcommand{\frob}[1]{\mathrm{Fr}_{#1}}

\newcommand{\mGal}[1]{\mathcal{G}_{#1}}

\newcommand{\mc}[1]{\mathcal{#1}}
\newcommand{\ol}[1]{\overline{#1}}
\newcommand{\Tr}[1]{\mathrm{Tr}_{#1}}

\usepackage{enumitem}
\usepackage{verbatim}
\begin{document}

\title{Recognizing Galois representations of K3 surfaces}
\author{Christian Klevdal}
\thanks{The author was partially supported by NSF DMS 1246989}
\address{Department of Mathematics\\
  University of Utah\\
 Salt Lake City, Utah, 84102.}
\email[C.~Klevdal]{klevdal@math.utah.edu}

\maketitle

\begin{abstract}
Under the assumption of the Hodge, Tate and Fontaine-Mazur conjectures we give a criterion for a compatible system of $\ell$-adic 
representations of the absolute Galois group of a number field to be isomorphic to the second cohomology of a K3 surface. This is 
achieved by producing a motive $M$ realizing the compatible system, using a local to global argument for quadratic forms to produce a K3 
lattice in the Betti realization of $M$ and then applying surjectivity of the period map for K3 surfaces to obtain a complex K3 surface. 
Finally we use a very general descent argument to show that the complex K3 surface admits a model over a number field. 
\end{abstract}

\section{Introduction}
This paper grew out of an attempt to answer a question on the section conjecture for moduli spaces of K3 surfaces, inspired by recent 
work of Patrikis, Voloch and Zarhin \cite{PVZ}. In this paper, the authors study the section conjecture for the moduli space of principally 
polarized abelian varieties. The section conjecture for a (geometrically connected) variety $X$ over a number field $K$ relates the set of 
rational points $X(K)$ with the sections of the fundamental sequence
	\[ 1 \to \pi_1(\ol{X}) \to \pi_1(X) \to \Gamma_K \to 1 \]
(we omit base points for the \'etale fundamental group and, for any field $K$, write $\Gamma_K = \mathrm{Gal}(\ol{K}/K)$ for the 
absolute Galois group with a fixed algebraic closure $\ol{K}$, and $\ol{X}$ the basechange of $X$ to $\ol{K}$). 
Given a rational point $x \colon \mathrm{Spec}(K) \to X$, functoriality of $\pi_1$ gives a section $\Gamma_K \to \pi_1(X)$ and this defines a 
map $\sigma_X \colon X(K) \to H(K, X)$ where $H(K,X)$ is the set of sections up to conjugation by $\pi_1(\ol{X})$. The section conjecture 
for $X$ states that the map $\sigma_X$ is a bijection. Of course, for general $X$ this map is far from a bijection, so we would want to find a 
class of varieties suitably determined by their fundamental groups. These are the so called anabelian varieties introduced by Grothendieck 
in his letter to Faltings \cite{Grothendieck}. Grothendieck suggested that hyperbolic curves, moduli spaces of curves and (less emphatically) 
moduli spaces of abelian varieties should all be anabelian. 

It is known that moduli spaces $\mc{A}_g$ of abelian varieties should not be anabelian by results of Ihara and Nakamura \cite{Ihara}. 
However, theorem 1.1 of \cite{PVZ} shows that under the assumption of well known motivic conjectures, a large subset of sections 
$S_0(K,\mc{A}_g) \subset H(K,\mc{A}_g)$ is contained in the image of $\sigma_{\mc{A}_g}$, where the sections $S_0(K, \mc{A}_g)$ are 
those coming from points locally. The authors are able to prove this by reducing to a question about Galois representations. More 
specifically there is a short exact sequence 
\begin{equation}\label{AbVarSeq}
\begin{tikzcd}
1 \arrow[r] & \pi_1(\ol{\mc{A}_g})  \arrow[r] \arrow{d}{\cong} & \pi_1(\mc{A}_g) \arrow[r]  \arrow[d] & \Gamma_K \arrow[r]  \arrow[d] & 1 \\
1 \arrow[r] & \mathrm{Sp}_{2g}(\hat{\ZZ})  \arrow[r] & \mathrm{GSp}_{2g}(\hat{\ZZ})  \arrow[r] & \hat{\ZZ}^\times  \arrow[r] & 1 
\end{tikzcd}
\end{equation}
Given a section $s \colon \Gamma_K \to \pi_1(\mc{A}_g)$ composition with the middle arrow gives a collection of $\ell$-adic 
representations $\{ \rho_\ell \colon \Gamma_K \to \mathrm{GSp}_{2g}(\ZZ_\ell) \}_{\ell}$. The fact that the left arrow is an isomorphism 
shows that the sections $H(K, \mc{A}_g)$ are determined by their associated $\ell$-adic representations. Then the authors use well known 
conjectures to find conditions on a collection of $\ell$-adic representations $\{ \rho_\ell \}$ that ensure they are isomorphic to the $\ell$-adic 
Tate module of an abelian variety \cite[Thm 3.3]{PVZ}. The proof of theorem 3.3 proceeds by using these conjectures to find a motive 
underlying the collection of $\ell$-adic representations. Taking Betti realization of this motive gives a Hodge structure that has the Hodge 
weights of an abelian variety. Using Riemann's theorem one can show that this Hodge structure is isomorphic to the Hodge structure on the 
Tate module of an abelian variety. 

One might ask whether \cite[Thm 3.3]{PVZ} can be generalized to other classes of varieties. In order for the above method to work, such a 
class of varieties would require an analogue of Riemann's theorem, which gives a criterion for an abstract Hodge structure to appear in the 
(co)homology of a variety. After abelian varieties, the most natural class of varieties with this property is K3 surfaces, where surjectivity of 
the period map is known. Our main theorem is precisely the analogue for K3 surfaces of \cite[Thm 3.3]{PVZ}. We find a set of conditions on 
a weakly compatible system of $\ell$-adic representations (for all $\ell$) to relate them to the weakly compatible system 
$H^2(X_{\ol{K}},\QQ_\ell)$ for varying primes $\ell$. We now make this precise.

Let $K$ be a number field and $S$ a finite set of rational primes. A collection of Galois representations $\{ \rho_\ell \colon \Gamma_K 
\to \GL{n}(\QQ_\ell)\}_{\ell \not\in S}$ is said to be weakly compatible if there exists a finite set $\Sigma$ of finite places of $K$ satisfying
\begin{enumerate}[I.]
	\item For each $\ell \not\in S$ the representation $\rho_\ell$ is unramified outside $\Sigma_\ell \cup \Sigma$ where $\Sigma_\ell$ 
	is the set of places of $K$ lying over $\ell$. 
	\item For each $\ell \not\in S$ and each place $v$ of $K$ not in $\Sigma_\ell \cup \Sigma$, the characteristic polynomial of 
	$\rho_\ell(\frob{v})$ has rational coefficients and is independent of $\ell$ (here $\frob{v}$ is a geometric Frobenius element at $v$). 
\end{enumerate}
Now let $\Lambda = U^{\oplus 3} \oplus E_8(-1)^{\oplus 2}$ be the K3 lattice. Fix a basis $e, f$ of the first copy of the hyperbolic plane 
$U$ such that $(e)^2 = (f)^2 = 0$ and $(e,f) = 1$. We prove the following

\begin{thm}\label{main}
Let $K$ be a number field. Assume the Hodge, Tate and Fontaine-Mazur conjectures. Let 
$\{ \rho_\ell \colon \Gamma_K \to O(\Lambda \otimes \QQ_\ell) \}$ be a weakly compatible system (with $S$ empty) of semisimple 
representations such that 
\begin{enumerate}
	\item There exist an integer $d > 0$ such that for all but finitely many primes $\ell$, $(e + df) \otimes 1 \in (\Lambda \otimes \QQ_\ell(1))^{\Gamma_K}$  
	\item For some $\ell_0$, $\rho_{\ell_0}$ is de Rham at all $v | \ell_0$.  
	\item For some $\ell_1$, $\mathrm{End}_{\Gamma_K}(\rho_{\ell_1}) = \QQ_{\ell_1} \oplus \QQ_{\ell_1}$.
	\item For some $\ell_2$ and some $v | \ell_2$, $\rho_{\ell_2}|_{\Gamma_{K_v}}$ is de Rham with Hodge-Tate weights $0,1,2$ 
		of multiplicities $1,20,1$. 
\end{enumerate}
Then there is a K3 surface $X$ over a finite extension $L$ of $K$ such that $\rho_\ell|_{\Gamma_{L}} \cong H^2(X_{\ol{L}}, \QQ_\ell)$ for all $\ell$. 
\end{thm}
The primes $\ell_0, \ell_1, \ell_2$ could all be the same. Conditions $(1), (2)$ and $(4)$ are of course necessary conditions for the 
collection $\{ \rho_\ell \}$ to be isomorphic to $H^2(X_{\ol{K}}, \QQ_\ell)$ for $X$ a K3 surface with a polarization of degree $2d$. 
Condition $(3)$ is an irreducibility condition (similar to the hypothesis of absolute irreducibility in \cite[Thm 3.3]{PVZ}) and is satisfied by 
the cohomology of the generic K3 surface, i.e.\ those of geometric Picard rank 1 with $\mathrm{End}_{\Gamma_K}(H^2(X_{\ol{K}}, 
\QQ_\ell)) \cong \QQ_\ell \oplus \QQ_\ell$. Suppose $X$ is a K3 surface over $K \subset \CC$. Recall that for any complex K3 surface there 
is a decomposition of rational $\QQ$-Hodge structures
	\[ H^2(X_\CC, \QQ) = (\mathrm{NS}(X_\CC) \otimes \QQ) \oplus (T(X_\CC) \otimes \QQ) \]
of the rational $\QQ$-lattice $\mathrm{NS}(X_\CC) \otimes \QQ$ and its orthogonal complement $T(X_\CC) \otimes \QQ$. A theorem of 
Zarhin \cite[Thm 1.4.1]{zarhin1983hodge} shows that $T(X_\CC) \otimes \QQ$ is an irreducible $\QQ$-Hodge structure. The Mumford-Tate 
conjecture, which is known for K3 surfaces by Tankeev \cite{tankeev1991k3}, then shows that condition $(3)$ is satisfied by $H^2(X_{\ol{K}}, 
\QQ_\ell)$ when $T(X_\CC)\otimes\QQ$ is an absolutely irreducible $\QQ$-Hodge structure, and the $\Gamma_K$-action on 
	\[ \mathrm{NS}(X_{\ol{K}}) \otimes \QQ_{\ell_1} = \mathrm{NS}(X_\CC) \otimes \QQ_{\ell_1} \] 
is absolutely irreducible. 

Our final remark regards the last section of the paper. The main result of this section, lemma \ref{spreading out}, shows a rather general 
criterion for the rigid descent of a variety $X$ over an algebraically closed field $\Omega$ of characteristic zero to an algebraically closed 
subfield $k \subset \Omega$. 

%
\subsection{Questions}
In the recent preprint \cite{Baldi}, Baldi independently proves an analogue of theorem \ref{main} above for K3 surfaces whose Picard rank 
$\rho$ satisfies $12 \leq \rho < 20$. This is theorem 1.2 of \cite{Baldi} where he shows for representations 
	\[ \{ \rho_\ell \colon \Gamma_K \to \mathrm{Gl}_{22-\rho}(\QQ_\ell) \}_\ell \]
satisfying analogues of conditions $(2), (3)$ and $(4)$ above, there is a finite extension $L$ of $K$ such that $\rho_\ell |_{\Gamma_{L}}$ is 
isomorphic to $T(X_{\ol{L}})_{\QQ_\ell}$. The main question given our result and that of Baldi is whether we can take $L = K$ for the field of 
definition of the K3 surface $X$, or at least if the degree of the finite extension can be bounded.

In both Baldi's work and this paper, the arguments of \cite{PVZ} are immediately extended to get a $\QQ$-Hodge structure $V$ of K3 type 
from the collection of $\ell$-adic representations. The next key step, which is unique to the K3 case, is to produce a lattice inside this $\QQ
$-Hodge structure and to show that this lattice is isomorphic to a (sub-)polarized $\ZZ$-Hodge structure of $H^2(X, \ZZ)$ for some complex 
K3 surface $X$. In Baldi's paper, this is done by picking any lattice $T\subset V$ and then using an embedding theorem of Nikulin to get an 
embedding of $T$ into the K3 lattice $\Lambda = U^{\oplus3}\oplus E_8(-1)^{\oplus 2}$. Then one shows that the Hodge structure induced 
by $T$ on $\Lambda$ is of K3 type and hence, by surjectivity of the period map, isomorphic to $H^2(X,\ZZ)$ for a K3 surface $X$ over 
$\CC$. Baldi's requirement that $12 \leq \rho < 20$ is a consequence of the fact that this is the range for which the embedding theorem 
holds. For our proof, we show that $V \cong \Lambda \otimes_\ZZ \QQ$ as quadratic spaces and thus produce a Hodge structure on the K3 
lattice $\Lambda$ from which we can apply surjectivity of the period map. In light of our theorem and Baldi's theorem it is natural to ask 
whether such a theorem holds for `Picard rank' $2 \leq \rho \leq 11$, i.e.\ compatible system of representations $\{ \rho_\ell \colon 
\Gamma_K \to \mathrm{O}(\Lambda \otimes \QQ_\ell) \}$ which decompose as $\rho_\ell \cong V_\ell \oplus \QQ_\ell(-1)^{\oplus{\rho}}$ 
with $V_\ell$ irreducible (and satisfying conditions (2) and (4) of theorem \ref{main}). The author does not believe that the proofs of this 
paper or those of Baldi can be adapted to prove an analogous theorem for $2 \leq \rho \leq 11$. 

Our original motivation was to apply theorem \ref{main} to answer a question about the section conjecture for moduli spaces of K3 surfaces, 
as was done in \cite[Thm 1.1]{PVZ} for abelian varieties. The moduli space we are interested in is the space $\mc{F}_{2d}$, using the 
notation of \cite{Rizov}, classifying primitively polarized K3 surfaces of degree 2d. Using the following diagram
\begin{equation}
\begin{tikzcd}
1 \arrow[r] & \pi_1(\ol{\mc{F}_{2d}})  \arrow[r] \arrow{rd}{\alpha} & \pi_1(\mc{F}_{2d}) \arrow[r]  \arrow[d] & \Gamma_K \arrow[r]  & 1 \\
& & \mathrm{O}({\Lambda_{2d}}\otimes\hat{\ZZ})  & & 
\end{tikzcd}
\end{equation}
we can associate to each section $s \in H(\mc{F}_{2d}, K)$ an $O( \Lambda_{2d} \otimes \hat{\ZZ})$ representation. If we knew that this 
map was a bijection then an analogue of \cite[Thm 1.1]{PVZ} could be proven. However, a computation of the group 
$\pi_1(\ol{\mc{F}_{2d}})$ seems difficult. One approach might be to compute the topological fundamental group of the complex analytic 
space $\mc{F}_{2d}^\text{an}$ and then compare the profinite completion to a suitable orthogonal group. The domain 
$\mc{F}_{2d}^\text{an}$ has an explicit description as the quotient by an orthogonal group of the complement of an infinite union of 
hyperplane sections in a period domain $D_d$, see \cite[Remark 6.3.7]{Huybrechts}. Given this explicit description, one may be able to 
compute the topological fundamental group. This is expected to be very large, containing an infinitely generated free group generated by 
loops around the hyperplane sections. These are the sorts of groups that are not residually finite, and the kernel to the profinite completion 
can be very large, see \cite{Toledo}. So while $\pi_1^\text{top}(\mc{F}_{2d}^\text{an})$ is very far from any orthogonal group, it may happen 
that the `non-orthogonal' part gets killed in the profinite completion. 

Finally, the last question that naturally follows from this is paper is whether there are analogues of theorem \ref{main} for hyperkahler 
varieties. There are known results on surjectivity of the period map and global Torelli theorems are known for hyperkahler manifolds, see 
\cite{Huybrecht-hyperkahler}, so it may be reasonable that methods in this paper could work for such varieties. 

%
\subsection{Terminology}
Throughout the paper, $K$ will be a number field with a fixed algebraic closure $\ol{K}$. We write $\Gamma_K = \mathrm{Gal}(\ol{K}/K)$ 
for the absolute Galois group. 

If $K$ is a field, and $F$ a field of characteristic $0$, we denote by $\cM{K,F}$ the category of pure numerical motives over $K$ with 
coefficients in $F$. If $F = \QQ$ we simply write $\cM{K}$. The functors $H_\ell, H_B, H_{dR}$ are the $\ell$-adic, Betti and algebraic 
de Rham realization functors. Implicitly whenever we write any of these functors, we are assuming the conjecture that numerical 
equivalence is equal to homological equivalence for that cohomology theory, and in this way the realization functors may be defined 
on numerical motives.

See section \ref{Quadratic forms and lattices} for notation about quadratic forms and lattices. If $V$ and $W$ are either both 
$\ZZ$-Hodge structures equipped or $\QQ$-Hodge structures with pairings (e.g. polarizations) then a map $V \to W$ is called a Hodge 
isometry if it is an isomorphism of Hodge structures that also respects the pairings. All K3 surfaces are smooth and projective. We write 
$\QQ\text{-HS}$ for the category of $\QQ$-Hodge structures. 

%
\subsection{Acknowledgements}
It is a pleasure to thank Stefan Patrikis, for suggested this problem to me, for his patient guidance and for the many helpful 
discussions we had. I would also like to thank Domingo Toledo for some helpful discussions. 

\section{Quadratic forms and lattices}\label{Quadratic forms and lattices}
%
\subsection{Notation}
If $F$ is a field of characteristic zero, a quadratic space over $F$ consists of a vector space $V$ with a non degenerate symmetric bilinear 
pairing $V \otimes V \to F$ (or equivalently a quadratic form on $V$). Of interest to us are quadratic spaces over 
$\QQ, \QQ_p$ and $\mathbb{R}$. When we talk about a lattice $T$, we mean a finitely generated free abelian group $T$ with a pairing 
$(\cdot, \cdot) \colon T\times T \to \ZZ$ that is non degenerate. 

%
\subsection{Lattices associated to K3 surfaces}
We write $U$ to denote the hyperbolic plane and $E_8$ the lattice associated to the Dynkin diagram $E_8$.

\begin{exmpl}\label{K3 lattice}
Let $\Lambda$ be the lattice $U^{\oplus 3} \oplus E_8(-1)^{\oplus 2}$. The disciminant $d(\Lambda) = -1$ so $\Lambda$ is even, 
unimodular and has signature $(3,19)$. This is called the K3 lattice, because if $X$ is a K3 surface over $\CC$, then the singular 
cohomology $H^2(X, \ZZ)$ with the cup product pairing is isomorphic to $\Lambda$. See \cite{Huybrechts}. 
\end{exmpl}

We will need the following lemma for the proof of the main theorem. 
\begin{lem}\label{k3 lattice}
Let $\Lambda$ be the K3 lattice. Then there exists a $\QQ$-quadratic space $V$ of signature $(r,s)$ such that 
$V \otimes_\QQ \QQ_p \cong \Lambda\otimes_\ZZ \QQ_p$ for all finite primes $p$ if and only if $(r,s)$ is one of the following pairs: 
	\[ (19,3), (15,7), (11,11), (7, 15), (3,19) \]
\end{lem}
\begin{proof}
By \cite[Thm 1.3, pg 77]{Cassels} and the fact that we know 
that the local data, we compute
\[\prod_{p\neq \infty} c_p(V) = \prod_{p\neq \infty} c_p(\Lambda) = c_\infty(\Lambda) = -1\]
(here we use theorem 1.3, page 77; theorem 1.2, page 56 loc.\ cit.). Again those same two theorems and the fact that we know the local 
data imply that $(-1)^{s(s-1)/2} = -1$, so $s \equiv 2,3 \mod 4$. By the Grunwald-Wang theorem, we can assume that $d(V) = d(\Lambda)$. 
Therefore by \cite[Thm 1.2, pg 56]{Cassels} we have $(-1)^s = -1$. We conclude from theorem 1.3 page 77 loc.\ cit.\ that the required $V$ 
will exist if and only if $s \equiv 3 \mod 4$. Therefore the possible signatures are $(19,3), (15,7), (11,11), (7, 15), (3,19)$. 
\end{proof}

\section{K3 surfaces}
\subsection{Facts about K3 surfaces}
For convenience of the reader, we recall facts about K3 surfaces that we will use. All proofs may be found in the book 
\cite{Huybrechts} whose terminology we use. 

Let $V$ be a finite free $\ZZ$ or $\QQ$-module. A Hodge structure of K3 type on $V$ is a weight 2 Hodge structure such that $V^{2,0}, 
V^{0,2}$ are $1$-dimensional and $V^{i,j} = 0$ for $|i-j| > 2$. Let $\Lambda$ be the K3 lattice (see example \ref{K3 lattice}). The period 
domain $D$ is defined as 
	\[ D := \{ x \in \mathbb{P}(\Lambda_\CC) \colon (x, \ol{x}) > 0, \quad (x)^2 = 0 \} \]
Given an element $x \in D$ we get a unique polarizable Hodge structure of K3 type on $\Lambda$ satisfying $\Lambda^{2,0} = x$ and 
$x \perp \Lambda^{1,1}$. The key fact that we use about K3 surfaces is the surjectivity of the period map, which we recall here. 

\begin{thm}[Surjectivity of the period map, {\cite[Ch.\ 6, Rmk 3.3]{Huybrechts}}]\label{surjectivity of period map}
For any $x \in D$ there exists a K3 surface $X$ and a Hodge isometry 
$\varphi \colon H^2(X, \ZZ) \xrightarrow{\sim} \Lambda$ such that $\varphi^{-1}(x)$ spans $H^{2,0}(X)$.  
\end{thm}

Following the definition of \cite{Huybrechts-isogeny}, two complex K3 surfaces $X$ and $X'$ are isogenous if there exists a $\QQ$-Hodge
isometry $H^2(X, \QQ) \xrightarrow{\sim} H^2(X', \QQ)$. We record the following lemma to be used later. 

\begin{lem}\label{isogeny-class-corrected}
Let $X$ be a complex K3 surface, and let $\mc{S}(X)$ be the set of isomorphism classes of complex K3 surfaces $Y$ that are isogenous to 
$X$. Then $\mc{S}(X)$ is countable. 
\end{lem}
\begin{proof}
The lemma will follow from a theorem of Huybrechts along with the countability of the Brauer group of a surface and finiteness results on 
Fourier-Mukai partners of twisted K3 surfaces. Recall that a twisted K3 surface $(S, \alpha)$ consists of a K3 surface $S$ and an element 
$\alpha$ of the Brauer group of $S$. The set of twisted Fourier-Mukai partners of $(S, \alpha)$ is 
	\[ \mathrm{FM}(S, \alpha) = \{(S', \alpha') \mid \text{ there is an equivalence } D^b(S, \alpha) \cong D^b(S', \alpha') \}/\cong \]
where $D^b(S, \alpha)$ is the derived category of $\alpha$-twisted coherent sheaves on $S$. All that is important for us is that for any K3 
surface $S$ and any $\alpha \in \mathrm{Br}(X)$ the set $\mathrm{FM}(S, \alpha)$ is countable, see \cite[Prop 4.3]{ma2010twisted} 

Let $Y \in \mc{S}(X)$. Then by \cite{Huybrechts-isogeny}, we can find Brauer classes $\alpha \in \mathrm{Br}(X)$ and $\beta \in 
\mathrm{Br}(Y)$, complex K3 surfaces $S_1, \ldots, S_n$ and Brauer classes $\alpha_i, \beta_i \in \mathrm{Br}(S_i)$ for $i = 1, \ldots, n$  
such that there is a chain of equivalences
\[ (X, \alpha) \sim_\text{FM} (S_1, \alpha_1), (S_1, \beta_1) \sim_\text{FM} (S_2, \alpha_2), \ldots, (S_{n-1}, \beta_{n-1}) \sim_\text{FM} 
(S_n, \alpha_n), (S_n, \beta_n) \sim_\text{FM} (Y, \beta) \]
Where we use $\sim_\text{FM}$ to denote the relation of being a twisted Fourier-Mukai partner. Further, we may assume $n \leq 22$. Recall 
that $\mathrm{Br}(S_i) = H^2(S_i, \mc{O}_{S_i}^\times)_\text{tors}$. From the exponential sequence, we get an exact sequence
	\[ 0 \to H^1(X, \mc{O}_{S_i}^\times) \to H^2(S_i, \ZZ) \to H^2(S_i, \mc{O}_{S_i}) \to H^2(S_i, \mc{O}_{S_i}^\times) \to 0 \]
and hence $H^2(S_i, \mc{O}_{S_i})$ is isomorphic as groups to $\CC/\ZZ^{\oplus 22 - \rho(S_i)}$. It follows that $\mathrm{Br}(S_i)$ is 
isomorphic to $(\QQ/\ZZ)^{\oplus 22-\rho(S_i)}$ which is countable. 
Hence, given $X$ there are only countably many options for $\alpha$ and by countability of $\mathrm{FM}(X, \alpha)$ only countably many 
options for $(S_1, \alpha_1)$. Likewise, there are only countably many options for $\beta_1$ and hence only countably many options for 
$(S_2, \alpha_2)$. Proceeding in this fashion we conclude there are only countably many options for $Y$. 
\end{proof}

\section{Motivic setup}
We recall basics facts about motives and refer the reader to \cite{Andre} for details. For a field $K$, write $\mc{P}_K$ for the category of 
smooth projective varieties over $K$. If $F$ is a field of characteristic zero we write $\cM{K,F}$ for the category of pure homological motives 
over $K$ with coefficients in $F$. There is a functor $h \colon \mc{P}_K \to \cM{K,F}$ that functions as a universal cohomology theory, 
meaning that if $H^\ast \colon \mc{P}_K \to F$-Alg is a Weil cohomology theory, then $H^\ast$ extends uniquely 
through $h$ to a functor $\cM{K,F} \to F$-Alg. Under the conjecture that numerical equivalence is the same as homological equivalence 
then the category $\cM{K,F}$ is a semisimple rigid abelian tensor category by \cite{Jannsen}. A choice of a Weil cohomology theory 
$H^\ast$ that extends to $\cM{K,F}$ is a fiber functor, making $\cM{K,F}$ a neutral Tannakian category. 
Thus by general theory, we have an equivalence between $\cM{K,F}$ and $\mathrm{Rep}_{\mGal{K,F}}(F)$, the category of $F$ 
representations of the pro-reductive algebraic group $\mGal{K,F} = \mathrm{Aut}^{\otimes} H^\ast$. We recall the most basic examples of 
fiber functors, and the extra structures they carry. 

\begin{exmpl}
Let $K$ be any field, $K^\text{sep}$ be a separable closure and $\ell$ a prime. For $X$ smooth projective over $K$, the $\ell$-adic 
cohomology $H_\ell(X) = H_{\text{et}}^\ast(X_{K^\text{sep}}, \QQ_\ell)$ is a Weil cohomology theory on $\mc{P}_K$. Further, $H_\ell(X)$ 
has a natural $\Gamma_K = \mathrm{Gal}(K^\text{sep}/K)$-action, and we write 
$H_\ell \colon \cM{K, \QQ_\ell} \to \mathrm{Rep}_{\QQ_\ell}(\Gamma_K)$ for the enriched $\ell$-adic realization functor. 
The Tate conjecture asserts that $H_\ell$ is fully faithful when $K$ is a number field. 
\end{exmpl}

\begin{exmpl}
Let $K = \CC$. For a smooth projective variety $X$ over $\CC$ we can form the corresponding complex-analytic manifold $X^\text{an}$. 
Singular cohomology $H_B(X) = H^\ast_\text{sing}(X^\text{an}, \QQ)$ is a Weil cohomology theory on $\mc{P}_\CC$. Further, $H_B(X)$ 
has a $\QQ$-Hodge structure, and we write $H_B \colon \cM{\CC} \to \QQ$-HS for the enriched Betti realization functor. The Hodge 
conjecture asserts that $H_B$ is fully faithful. 
\end{exmpl}

\begin{exmpl}
Let $K$ be a field of characteristic $0$. For a smooth projective variety $X$ over $K$ we have the algebraic de Rham complex 
$\Omega_{X/K}^\bullet$. Algebraic de Rham cohomology $H_{dR}(X) = \mathbb{H}^\ast(X,\Omega_{X/K}^\bullet)$ is a Weil cohomology 
theory on $\mc{P}_K$ (with coefficients in $K$). Further, $H_{dR}(X)$ has a filtration, and we write 
$H_{dR} \colon \cM{K,K} \to \mathrm{Fil}_K$ (with $\mathrm{Fil}_K$ the category of filtered $K$-vector spaces) for the corresponding 
enriched de Rham realization functor. 
\end{exmpl}

For a given embedding $\iota \colon \overline{\QQ} \to \overline{\QQ_\ell}$ let $H_\iota$ be the composition  
$\cM{K, \overline{\QQ}} \to \cM{K, \overline{\QQ_\ell}} \to \mathrm{Rep}_{\overline{\QQ_\ell}}(\Gamma_K)$. The following lemma is taken 
from \cite{PVZ}, with a slight weakening of the hypothesis due to \cite{Moonen}, in which it is shown that the Tate conjecture implies the 
Grothendieck-Serre semisimplicity conjecture. 

\begin{lem}[Lemma 3.3, \cite{PVZ}]\label{3.3pvz}
Assume the Tate and Fontaine-Mazur conjectures, and let $K$ be a number field. If $r_\ell \colon \Gamma_K \to \GL{N}(\QQ_\ell)$ is an 
irreducible geometric Galois representation. Then there exists an object $M$ of $\cM{K, \overline{\QQ}}$ such that 
$r_\ell \otimes_{\QQ_\ell} \overline{\QQ_\ell} \cong H_\iota(M)$.
\end{lem}

\section{Proof of main theorem}
The first part of the proof follows closely that of \cite[Thm 3.1]{PVZ}, the main difference being that we have to worry about carrying the 
bilinear form through the motivic yoga. 
\begin{proof}[Proof of theorem \ref{main}] 
Let $\{ \rho_\ell \colon \Gamma_K \to \mathrm{O}(\Lambda \otimes \QQ_\ell) \}$ be as in the theorem. Fix an embedding 
$\iota_0 \colon \overline{\QQ} \hookrightarrow \overline{\QQ_{\ell_0}}$. Then from lemma \ref{3.3pvz} (which extends to 
semisimple geometric representations) we have a motive $M$ in $\cM{K,\overline{\QQ}}$ such that 
$H_{\iota_0}(M) \cong \rho_{\ell_0}\otimes \overline{\QQ_{\ell_0}}$. In fact, $M$ has coefficients in some finite extension 
$E$ of $\QQ$ inside $\ol{\QQ}$. Let $\rho \colon \mGal{K,E} \to \GL{22, E}$ be the associated motivic Galois representation. Now 
$\iota_0$ induces some place $\lambda_0$ of $E$. If $\lambda$ is a finite place of $E$ (say $\lambda \mid \ell$) let $\rho_\lambda$ be 
the $\lambda$-adic realization of $\rho$. Then as in \cite{PVZ}, for almost all places $v$ of $K$ there is an equality of the rational numbers
	\[ \tr(\rho_\lambda(\frob{v})) = \tr(\rho_{\lambda_0}(\frob{v})) = \tr(\rho_{\ell_0}(\frob{v})) = \tr(\rho_{\ell}(\frob{v})). \]
By Brauer-Nesbitt, a continuous semisimple Galois representations $\rho$ of $\Gamma_K$ is determined by $\tr(\rho(\sigma))$ for 
$\sigma$ in a dense subset of $\Gamma_K$. By Chebotarev we may take the collection $\frob{v}$ for $v$ as above and conclude that
	\[ \rho_\lambda \cong \rho_\ell \otimes_{\QQ_\ell} E_\lambda\] 
for all $\lambda$. Conditions $(1)$ and $(3)$ of our assumption say for a place $\lambda_1 \mid \ell_1$ that 
$\rho_{\ell_1} \otimes E_{\lambda_1}(1)$ splits as a sum of the trivial representation and an absolutely irreducible representation. By the 
Tate conjecture, we conclude that $\rho = \mathbf{1}(-1) \oplus \rho'$ (where $\mathbf{1}(-1)$ is the Tate twist of $\mathbf{1}$) with $\rho'$ 
absolutely irreducible. It follows that each $\rho_\ell$ is isomorphic to $\QQ_\ell(-1) \oplus V_\ell$ with $V_\ell$ 
an absolutely irreducible representation of $\Gamma_K$.

Lemma 3.4 of \cite{PVZ} shows that $\rho'$ descends to $\QQ$, and clearly $\mathbf{1}(-1)$ does, hence $\rho$ descends to $\QQ$. 
We have $\Gamma_K$-equivariant pairings $\rho_\ell \otimes \rho_\ell \to \QQ_\ell$, and thus there are pairings 
$\rho_\lambda \otimes_E \rho_\lambda \to E_\lambda $. By the Tate conjecture, there is an isomorphism
	\[ \left(\mathrm{Sym}^2 \rho_E^\vee\right)^{\mGal{K,E}}\otimes_E E_\lambda \cong \left(\mathrm{Sym}^2 \rho_\lambda^\vee
	\right)^{\Gamma_K} \]
Hence we get a non degenerate $\mGal{K,E}$-equivariant pairing $\rho_E \otimes_E \rho_E \to E$. However each local pairing descends 
to $\QQ_\ell$. By Galois descent the map
	\[ \Hom_{\mGal{K}}(\rho\otimes_\QQ \rho, \QQ) \to \Hom_{\mGal{K,E}}(\rho_E \otimes \rho_E, E)^{\Gamma_\QQ} \]
is an isomorphism, and therefore the pairing on $\rho_E$ descends to a pairing $\rho\otimes_\QQ \rho \to \QQ$. 

Now that we have a motivic Galois representation $\rho \colon \mGal{K} \to \GL{22,\QQ}$ whose $\ell$-adic realizations are 
$\rho_\ell$, we can do some comparisons. Let $M \in \cM{K}$ be the corresponding rank 22 motive. 
Objects of $\cM{K}$ enjoy the de Rham comparison theorem of $p$-adic Hodge theory. In particular, for $v$ and $\ell_2$ as in condition $(4)$ 
there are isomorphisms:
	\[ H_{dR}(M) \otimes_K B_{dR, K_v} \xrightarrow{\sim} H_{\ell_2}(M) \otimes_{\QQ_{\ell_2}} B_{dR, K_v} \]
Hence 
	\[ H_{dR}(M)\otimes_K K_v \cong D_{dR, K_v}(H_{\ell_2}(M)) = \left(H_{\ell_2}(M) \otimes_{\QQ_{\ell_2}} B_{dR, K_v}\right)^{\Gamma_{K_v}} \]
By assumption $(4)$ and the comparison isomorphism, the Hodge filtration on $H_{dR}(M)$ satisfies
	\[ \dim_K \mathrm{gr}^i H_{dR}(M) = \begin{cases} 1 & \text{if } i = 0,2, \\ 20 & \text{if } i = 1, \\ 0 & \text{otherwise.} \end{cases} \]
The Betti de-Rham comparison theorem states that 
	\[ H_{dR}(M|_\CC) \cong H_B(M) \otimes_\QQ \CC, \]
so $H_B(M)$ is a Hodge structure of K3 type. It is also a $\QQ$-quadratic space, coming from the motivic pairing. We will show that 
there is an isomorphism of $\QQ$-quadratic spaces $H_B(M) \cong \Lambda \otimes \QQ$ with $\Lambda$ the K3 lattice 
$U^{\oplus 3} \oplus E_8(-1)^{\oplus 2}$. First, note that we have comparison isomorphisms $H_B(M) \otimes \QQ_\ell \cong H_\ell(M)$ 
that respect the pairings on both spaces. By assumption $H_\ell(M) \cong \Lambda \otimes \QQ_\ell$ as quadratic spaces, so to show 
that $H_B(M) \cong \Lambda \otimes \QQ$ it is enough to show that $H_B(M)$ has the same signature as $\Lambda$, which is 
$(3, 19)$. Now $M = \mathbf{1}(-1) \oplus M'$ with $M'$ absolutely irreducible. Thus we have an orthogonal decomposition of Hodge 
structures $H_B(M) = \QQ(-1) \oplus H_B(M')$ with $H_B(M')$ irreducible.

We compute the possible signatures on $H_B(M')$. First, there is a smooth projective $X$ such that $M' \hookrightarrow h^k(X)(j)$ 
for some integers $k,j$ where $h^k(X)(j)$ is the motive whose realization is $H^k(X)(j)$. The pairing on $M'$ is up to a scalar 
multiple the same as the intersection pairing coming from a polarization $L$ on $X$ because 
$\dim_\QQ(\mathrm{Sym}^2 \rho'^\vee)^{\mGal{K}} = 1$. 
Since $M'$ has weight 2, we know that $k$ is even and $k - 2j = 2$. There is a decomposition of motives \cite[Prop 5.2.5.1]{Andre}
	\[ h^k(X) = \bigoplus_{r\leq k} L^r h^{k-2r}_\text{prim}(X)(-r)\]
so $M' \subset L^r h^{k-2r}_\text{prim}(X)(j-r)$ for some $r$ and thus $H_B^{1,1}(M') \subset L^rH^{j+1-r,j+1-r}_\text{prim}(X)$. The 
intersection pairing on this subspace is definite by the Hodge index theorem \cite[Thm 6.32]{Voisin}. By lemma \ref{k3 lattice} the only 
possible signatures of $H_B(M)$ are $(3,19), (7,15), (11,11), (15,7)$ and $(19, 3)$ but since the form is definite on the $19$-dimensional 
subspace $H^{1,1}(M') \subset H_B(M)\otimes \CC$, the signature must be $(3,19)$ or $(19,3)$. The quadratic form restricted to 
$H_B(\mathbf{1}(-1))$ is determined by an element $\alpha$ of $\QQ^\times/(\QQ^\times)^2$. Assumption $(1)$ assures that the image of 
$H_B(\mathbf{1}(-1))$ under the comparison isomorphism
	\[ H_B(\mathbf{1}(-1))\otimes \QQ_\ell \cong H_\ell(\mathbf{1}(-1)) \]
maps to the line spanned by the vector $(e + df) \otimes 1$ in $\Lambda \otimes \QQ_\ell$ (using the notation in theorem \ref{main}). 
The fact that $(e+df)^2 = 2d$ shows that the image of $\alpha$ in $\QQ_\ell^\times/(\QQ_\ell^\times)^2$ is $2d$ for almost all $\ell$. 
Thus by the Grunwald-Wang theorem, we have that $\alpha = 2d$. Consequently, there is an isomorphism of quadratic spaces 
$H_B(\mathbf{1}(-1)) \cong \QQ(2d)$ where the bilinear form on $\QQ(2d)$ is given by $(a,b) = 2dab$. In particular, the pairing on 
$H_B(\mathbf{1}(-1))$ is positive definite. We conclude that the signature on $H_B(M)$ is $(3,19)$ which completes the proof that 
$H_B(M) \cong \Lambda \otimes \QQ$ as quadratic spaces. Let $i \colon \Lambda \hookrightarrow H_B(M)$ be an embedding and write 
$\Lambda(M)$ for the image of $\Lambda$ under this embedding. 

Now $\Lambda(M)$ has an induced $\ZZ$-Hodge structure from that on $H_B(M)$. By surjectivity of the period map (theorem 
\ref{surjectivity of period map}) we know that there is a K3 surface $X$ over $\CC$ with a Hodge isometry 
$H^2(X,\ZZ) \cong \Lambda \subset H_B(M)$. The Hodge conjecture implies that $M|_\CC \cong h^2(X)$ and further that this 
isomorphism respects the pairings on each motive.  For each $\sigma \in \mathrm{Aut}(\CC/\ol{\QQ})$ we have
	\[ h^2(X)^\sigma \cong M|_\CC^\sigma = M|_\CC \cong h^2(X), \]
where $M|_\CC$ is the image of $M$ under base change $\mc{M}_K \to \mc{M}_{\CC}$  and $h^2(X)^\sigma, M|_\CC^\sigma$ are the 
$\sigma$-conjugates of $h^2(X)$ and $M|_\CC$. Since these isomorphisms respect the pairing, upon applying Betti realization we see that 
$X$ is isogenous to each conjugate $X^\sigma$, hence by corollary \ref{K3descent}, $X$ admits a model over $\ol{\QQ}$. We denote this 
model by $X_{\ol{\QQ}}$ and write $X_\CC$ for complex K3 surface above. The map 
	\[ \Hom_{\cM{\ol{\QQ}}}(h^2(X_{\ol{\QQ}}), M|_{\ol{\QQ}}) \to \Hom_{\cM{\CC}}(h^2(X_\CC), M|_\CC) \]
is an isomorphism, and hence there is an isomorphism $\alpha \colon h^2(X_{\ol{\QQ}}) \to M|_{\ol{\QQ}}$ in $\cM{\ol{\QQ}}$. Let $L$ be a 
finite extension of $K$ such that both $X$ and $\alpha$ are defined over $L$. Hence we get an isomorphism 
$\alpha \colon h^2(X) \to M|_L$ in $\cM{L}$ which yields isomorphisms 
	\[ H_\ell(\alpha) \colon H^2(X_{\ol{L}}, \QQ_\ell) \xrightarrow{\sim} H_\ell(M) = \rho|_{\Gamma_{L}} \otimes \QQ_\ell \]
\end{proof}

\section{Descent}
In this section we prove that complex K3 surfaces that are isogenous to all of their $\ol{\QQ}$-conjugates admit models over number fields. This will follow from a general spreading out argument. 
\begin{prop}
Let $L$ be an extension of a characteristic zero field $K$ and $E$ and $F$ subextensions of $L$. Suppose that 
\begin{enumerate}
\item $K$ is algebraically closed in one of $E$ or $F$. 
\item $E$ and $F$ are algebraically disjoint over $K$.
\end{enumerate}
Then $E$ and $F$ are linearly disjoint over $K$. 
\end{prop}
\begin{proof}
Let $\{x_i\}_{i\in I}$ be a transcendence basis for $E$ over $K$ and define $E_0 = K(\{x_i\}_{i\in I})$. Likewise let $\{y_j\}_{j \in J}$ be a transcendence basis for $F$ over $K$ and define $F_0 = K(\{y_j\}_{j\in J})$. 

First, notice that $E_0$ and $F_0$ are linearly disjoint over $K$. This follows from the fact that they are algebraically disjoint over $K$, 
and $E_0$ and $F_0$ are purely transcendental extensions of $K$ see by \cite[Ch.\ V, \S 14, prop.\ 14]{bourbaki2013algebra}. 
It also follows from the same proposition that $E$ and $F_0$ are linearly disjoint over $K$, as are $E_0$ and $F$. 
The theorem holds as long as $EF_0$ and $E_0F$ are linearly disjoint over $E_0F_0$, by \cite[Ch.\ V, \S 14, prop.\ 8]{bourbaki2013algebra}. In what follows, we will assume that $K$ is algebraically closed in $E$. 

We will show that $E \cap F = K$. We already know that $E_0 \cap F_0 = K$ as they are linearly disjoint over $K$. Any element 
$\alpha \in E \cap F$ is contained in a finite extension of $E_0$ and $F_0$ so we may assume for now that $E$ is finite over $E_0$ and 
$F$ is finite over $F_0$. We claim that $\Tr{E/E_0}(\alpha) \in K$. Indeed since $\alpha \in EF_0$ and $\alpha \in E_0F$ we have the 
following equalities
\begin{align*}
[EF\colon E_0F] \Tr{E_0F/E_0F_0}(\alpha) &= \Tr{E_0F/E_0F_0}(\Tr{EF/E_0F}(\alpha)) \\
&=\Tr{EF/E_0F_0}(\alpha) \\
&= \Tr{EF_0/E_0F_0}(\Tr{EF/EF_0}(\alpha)) \\
&= [EF\colon EF_0] \Tr{EF_0/E_0F_0}(\alpha) 
\end{align*}
Notice that $\Tr{EF_0/E_0F_0}(\alpha) = \Tr{E/E_0}(\alpha)$ as $EF_0 \cong E\otimes_{E_0} E_0F_0$ (as the two fields are linearly disjoint 
over $E_0$ by the first remarks, and $E$ is algebraic over $E_0$) and trace is invariant under extension of scalars. Therefore 
	\[ \Tr{E/E_0}(\alpha) = \Tr{EF_0/E_0F_0}(\alpha) = \frac{[EF\colon E_0F]}{[EF \colon EF_0]} \Tr{E_0F/E_0F_0} = 
	\frac{[EF\colon E_0F]}{[EF \colon EF_0]} \Tr{F/F_0}(\alpha) \in F_0 \]
and consequently $\Tr{E/E_0}(\alpha) \in E_0 \cap F_0 = K$. However, by the same reasoning $\Tr{E/E_0}(\alpha^n) \in K$ for $n = 0, 1, 2, 
\ldots, [E\colon E_0]$. From Newton's identities, we see that the minimal polynomial of $\alpha$ over $E_0$ has coefficients in $K$, and 
therefore $\alpha$ is algebraic over $K$. But from the assumption that $K$ is algebraically closed in $E$ we see that $\alpha \in K$, as 
required. 

Finally, to show that $EF_0$ and $E_0F$ are linearly disjoint over $E_0F_0$ we may enlarge $E$ and $F$ so that they are normal over 
$E_0$ and $F_0$ respectively. Then $EF_0$ and $E_0F$ are normal over $E_0F_0$ and hence it is enough to show that 
$EF_0 \cap E_0F = E_0F_0$. Suppose $x \in EF_0 \cap E_0F$ and write $x = ef_0 = e_0f$ for $e_0 \in E_0, e \in E, f_0 \in F_0$ and 
$f \in F$. Then $ee_0^{-1} = ff_0^{-1} \in E \cap F = K$. Therefore $e = (ee_0^{-1})e_0 \in KE_0 = E_0$ and therefore $x \in E_0F_0$ as 
required. 
\end{proof}

\begin{cor}\label{fieldtheory}
Let $\Omega$ be an extension of a field $k$ with $\Omega$ and $k$ algebraically closed of characteristic zero. Suppose that $\Omega$ 
has uncountable transcendence degree over $k$. Then for any subextension $E$ of $\Omega$ that is countably generated over $k$, there 
is an element $\sigma \in \mathrm{Aut}(\Omega/k)$ such that $E$ and $\sigma(E)$ are linearly disjoint over $k$
\end{cor}
\begin{proof}
Let $x_1, x_2, \ldots$ be a transcendence basis of $E$ over $k$, and let $y_1, y_2, \ldots $ be any elements of $\Omega$ such that the 
collection $\{x_1, x_2, \ldots, y_1, y_2, \ldots\}$ are algebraically independent over $k$. Let $\sigma \in \mathrm{Aut}(\Omega/k)$ be any 
element satisfying $\sigma(x_i) = y_i$. By construction the fields $E$ and $\sigma(E)$ are algebraically disjoint over $k$. As $k$ is 
algebraically closed, the previous theorem shows $E$ and $\sigma(E)$ are linearly disjoint over $k$. 
\end{proof}

\begin{lem}\label{spreading out}
Let $\Omega$ be an extension of a field $k$ with $\Omega$ and $k$ algebraically closed of characteristic zero, and $\Omega$ of 
uncountable transcendence degree over $k$. Let $X$ be a variety over $\Omega$ whose conjugates $X^\sigma$ for $\sigma \in 
\mathrm{Aut}(\Omega/k)$ are contained in a countable set. Then $X$ admits a model over a number field. 
\end{lem}
\begin{proof}
We can choose $\sigma_1, \sigma_2 \ldots \in \mathrm{Aut}(\Omega/k)$ so that every conjugate of $X$ is isomorphic to $X^{\sigma_i}$ 
for some $i$. For each $i$, $X^{\sigma_i}$ is defined over a finitely generated field extension $K_i$ over $k$. Let $K \subset \Omega$ be 
the composite in $\Omega$ of $K_1, K_2 \ldots$ so that $X^{\sigma_1}, X^{\sigma_2}, \ldots$ (hence any conjugate of $X$) admit models 
over $K$. Note that $K$ countably generated over $k$.  Let $\tau \in \mathrm{Aut}(\Omega/k)$ be any automorphism with $\tau(K)$ and 
$K$ linearly disjoint over $k$ which exists by corollary \ref{fieldtheory}. Suppose $X_0$ is a model of $X$ over $K_i$ for some $i$. As 
$X^\tau$ is isomorphic (over $\Omega$) to $X^{\sigma_j}$ for some $j$, we know that $X_0^\tau$ is a model of $X^{\sigma_j}$ over 
$\tau(K_i)$. Thus $X^{\sigma_j}$ admits a model over $K_j$ and $\tau(K_i)$. These are finitely generated and linearly disjoint over $k$, so 
$X^{\sigma_j}$ admits a model over $k$ see \cite[Lem.\ 16.28]{Milne}. Hence $X$ admits a model over $k$.
\end{proof}

\begin{cor}\label{K3descent}
If $X$ is a complex K3 surface and $X$ is isogenous to $X^\sigma$ for all $\sigma \in \mathrm{Aut}(\CC/\ol{\QQ})$ then $X$ admits a model 
over $\ol{\QQ}$.
\end{cor}
\begin{proof}
By assumption $X^\sigma \in \mc{S}(X)$ where $\mc{S}(X)$ is the isogeny class of $X$. From lemma \ref{isogeny-class-corrected}  
$\mc{S}(X)$ is countable and by the previous lemma it follows that $X$ has a model over $\ol{\QQ}$.
\end{proof}

\printbibliography
\end{document}